\documentclass[a4paper,12pt]{amsart} 

\usepackage{amssymb,amsmath}
\usepackage{pb-diagram}
\usepackage{pb-xy}
\usepackage[cmtip,arrow]{xy}

\newtheorem{thm}{Theorem}[section] 
\newtheorem*{thm*}{Theorem}
\newtheorem{prop}[thm]{Proposition}
\newtheorem{cor}[thm]{Corollary} 
\newtheorem{lem}[thm]{Lemma}

\theoremstyle{definition} 
 
\newtheorem{rem}[thm]{Remark} 
\newtheorem{exa}[thm]{Example}

\newtheorem*{ackn}{Acknowledgment}

\numberwithin{equation}{section}

\newcommand{\skal}[2]{\langle #1,#2\rangle}
\newcommand{\alg}[1]{\mathfrak{#1}}

\begin{document}

\title{Harmonic $SU(3)$-- and $G_2$--structures via spinors}
\author{Kamil Niedzia\l omski}

\subjclass[2010]{}
 
\address{
Department of Mathematics and Computer Science \endgraf
University of \L\'{o}d\'{z} \endgraf
ul. Banacha 22, 90-238 \L\'{o}d\'{z} \endgraf
Poland
}
\email{kamil.niedzialomski@wmii.uni.lodz.pl}

\begin{abstract}
In this note, using the spinorial description of $SU(3)$ and $G_2$--structures obtained recently by other authors, we give necessary and sufficient conditions for harmonicity of above mentioned structures. We describe obtained results on appropriate homogeneous spaces. Here, harmonicity means harmonicity of the certain unique section induced by the $G$--structure in consideration.
\end{abstract}

\maketitle

\section{Introduction}

Let $M$ be a spin manifold equipped with a $SU(3)$ (then $n=\dim M=6$) or $G_2$ (then $n=\dim M=7$) structure. This means, that the special orthonormal frame bundle $SO(M)$ has a reduction of a structure group $SO(n)$ to $SU(3)$ or $G_2$, respectively. Recently, the authors of \cite{ACFH} has studied such structures via sponorial approach. A global unit spinor $\varphi$ in the spinor bundle defines, depending on the dimension of a manifold, mentioned above structures. Thus it is natural to study the geometry of a defining spinor $\varphi$. The crucial observation in \cite{ACFH} is the existence of an endomorphism $S:TM\to TM$ and a one form $\eta$ (which vanishes in the $G_2$ case by the dimensional reasons) which describe the covariant derivative of $\varphi$,
\begin{equation*}
\nabla_X\varphi=S(X)\cdot\varphi+\eta(X)j\cdot\varphi, 
\end{equation*}
where $j$ is a certain almost complex structure on the spinor bundle.

On the other hand, C. M. Wood \cite{Wo,Wo2} and later Gonz\'{a}lez-D\'{a}vila and Martin Cabrera\cite{GDMC} introduced and studied so called harmonic $G$--structures. Each $G$--structure $P\subset SO(M)$ defines a section $\sigma_P$ of the associated bundle $SO(M)\times_G (SO(n)/G)$ by
\begin{equation*}
\sigma(x)=[p,eG],\quad \pi_{SO(M)}(p)=x.
\end{equation*}
In a compact case, if $\sigma$ is a harmonic section, then we say that the corresponding $G$--structure is harmonic. This condition is a differential equation $\sum_i (\nabla_{e_i}\xi)_{e_i}=0$ involving the intrinsic torsion $\xi$, the main ingredient of all considerations. This condition is treated as a harmonicity condition also in a non--compact case. The intrinsic torsion $\xi$, shortly speaking, is a $(2,1)$--tensor field being the difference of the Levi--Civita connection $\nabla$ and the $G$--connection $\nabla^G$ induced by the $\alg{g}$--component of the connection form of $\nabla$ (here $\alg{g}$ is the Lie algebra of $G$), $\xi_XY=\nabla_XY-\nabla^G_XY$. Thus the intrinsic torsion measures the defect of a $G$--structure to have holonomy in $G$.

In this note, the author tries to combine these two approaches for $G=SU(3)$ and $G=G_2$. The condition of harmonicity becomes a differential condition on $S$ and $\eta$. In some cases, for example $\eta=0$, it takes really simple form. Among others, we conclude that if a $SU(3)$--structure is in its $\mathcal{W}_1$ or $\mathcal{W}_3$ class from the Gray--Hervella classification of possible intrinsic torsion modules, then it  is harmonic. Analogously, if a $G_2$--structure is in $\mathcal{W}_1$ class with the defining function $\lambda$ being constant or in $\mathcal{W}_3$ class, then it is harmonic. 

Finally, we deal with some examples on homogeneous spaces. We begin by general considerations, which lead to a following conclusion -- if a unit spinor defining a $G$--structure is induced by a fixed point of a isotropy representation and the minimal $G$--connection is induced by a zero map $\Lambda_{\alg{g}}$ (see the last section for details), then considered structure is harmonic. We justify these results on appropriate examples. Although, these examples has been already considered in the literature, they have not been studied from this point of view. In other words, we find 'new' examples of harmonic $SU(3)$ and $G_2$--structures.    

\begin{ackn}
The author wishes to thank Ilka Agricola for fruitful conversations during his short visits to Marburg University, in particular, for the discussions concerning geometry of homogeneous spaces. He is also grateful for the notes concerning the geometry of the complex projective space $\mathbb{CP}^3$ and the remark about alternative approach to some of the obtained results (see Remark \ref{rem:Agricola}).\\
The author is supported by the National Science Center, Poland - Grant Miniatura 2017/01/X/ST1/01724.  
\end{ackn}

\section{The intrinsic torsion and spinors - A general approach}

Let $(M,g)$ be an oriented Riemannian manifold equipped with a $G$--structure, $G\subset SO(n)$, where $n=\dim M$. Let $\nabla^g$ be the Levi--Civita connection of $g$ and $\omega$ be its connection form on $SO(M)$. The splitting on the level of Lie algebras
\begin{equation}\label{eq:sonsplitting}
\alg{so}(n)=\alg{g}\oplus\alg{m},
\end{equation} where $\alg{m}$ is an orthogonal complement of $\alg{g}$ in $\alg{so}(n)$, defines a splitting $\omega=\omega_{\alg{g}}+\omega_{\alg{m}}$. By the fact that \eqref{eq:sonsplitting} is ${\rm ad}(G)$--invariant, it follows that $\alg{g}$--component $\omega_{\alg{g}}$ is a connection form on the $G$--reduction $P\subset SO(M)$ and hence defines a connection $\nabla^G$ on $M$. The difference
\begin{equation*}
\xi_XY=\nabla^G_XY-\nabla^g_XY,\quad X,Y\in\Gamma(TM),
\end{equation*}
defines a tensor $\xi\in T^{\ast}M\otimes\alg{m}(TM)$ called the {\it intrinsic torsion}. It follows immediately that its alternation is, up to a sign, the torsion $T^G$ of $\nabla^G$.

Denote by $N$ the associated bundle $SO(M)\times_{SO(n)}(SO(n)/G)$. There is one to one correspondence between $G$--structures $(M,g)$ and the sections of $N$. Thus, $P\subset SO(M)$ defines the unique section $\sigma_P\in\Gamma(N)$. We say that a $G$--structure $P$ is harmonic if $\sigma_P$ is a harmonic section \cite{GDMC} (if $M$ is compact). It can be shown \cite{GDMC} that harmonicity is equivalent to vanishing of the following tensor
\begin{equation}\label{eq:harmonicGstr}
L=\sum_i (\nabla^g_{e_i}\xi)_{e_i},
\end{equation}  
where the sum is taken with respect to any orthonormal basis. We treat this condition as harmonicity condition also in a non--compact case. Since, informally speaking, $\nabla^G$ preserves the decomposition \eqref{eq:sonsplitting} and by the fact that 
\begin{equation*}
L=\sum_i ((\nabla^G_{e_i}\xi)_{e_i}+\xi_{\xi_{e_i}e_i}),
\end{equation*}
we see that $L\in \alg{m}(TM)$ \cite{GDMC}. 

Assume $M$ is equipped with a spin structure. Let $\rho:{\rm Spin}(n)\to\Delta_n$ be the real spin representation. We will denote this action, and all induced actions, by a 'dot'. In particular acting on spinors, we have
\begin{equation}\label{eq:cliffordmultiplication}
X^{\flat}\cdot\omega-\omega\cdot X^{\flat}=2 X\lrcorner\,\omega.
\end{equation} 

Denote by $S(M)$ the induced spinor bundle, $S(M)={\rm Spin}(M)\times_{\rho}\Delta_n$, where ${\rm Spin}(M)$ is a spin structure. Assume $G$ is a stabilizer of some unit spinor $\varphi_0\in \Delta_n$ and let $\varphi\in S$ be the corresponding unit spinor, which defines a $G$--structure. Hence, with the usual identification of $\alg{so}(TM)$ with the space of $2$--forms on $M$ the action of $\alg{m}(TM)$ on $\varphi$ is injective. Therefore, vanishing of $L\in \alg{m}(TM)$ is equivalent to the relation
\begin{equation}\label{eq:harmonicGstrspinor}
L\cdot\varphi=0,
\end{equation}
where $\cdot$ denotes the action of skew--forms on spinors.

Let us describe above condition with the use of spinorial laplacian. Denote with the same symbol $\nabla^g$ the connection on $S(M)$ induced from the Levi--Civita connection $\nabla^g$. Then we put \cite{FI}
\begin{equation*}
\Delta\psi=-\sum_i (\nabla^g_{e_i}\nabla^g_{e_i}\psi-\nabla^g_{\nabla^g_{e_i}e_i}\psi).
\end{equation*}
We have \cite{ACFH} 
\begin{equation}\label{eq:mainstartrelation}
\nabla^g_X\varphi=\frac{1}{2}\xi_X\cdot\varphi
\end{equation}
which follows from the fact that $\nabla^G\varphi=0$. Differentiating \eqref{eq:mainstartrelation} we get
\begin{equation*}
\Delta\varphi=-\frac{1}{2}L\cdot\varphi-\frac{1}{4}\sum_i\xi_{e_i}\cdot \xi_{e_i}\cdot\varphi.
\end{equation*}
The second component on the right hand side can be interpreted in the following way. For any tensor $T\in T^{\ast}M\otimes\alg{so}(TM)$ define
\begin{equation*}
c_T=\frac{1}{2}\sum_i T_{e_i}\cdot T_{e_i}\quad\textrm{and}\quad \sigma_T=\frac{1}{2}\sum_i T_{e_i}\wedge T_{e_i}
\end{equation*}
as elements of Clifford bundle acting on spinors. These elements, defined for totally skew tensors, has been already considered and theirs important role have been established \cite{Ag0}. Then (we do not require $T$ to be totally skew--symmetric)
\begin{equation}\label{eq:cTandsigmaT}
c_T=\sigma_T-\frac{3}{2}|T|^2.
\end{equation}
The above equation shows that $c_T$ acting on spinors contains element of second order, namely $\sigma_T$, and a scalar $|T|^2$. We have proved the following general characterization of harmonic $G$--structures defined by a spinor.
\begin{prop}\label{prop:general}
A $G$--structure on a spin manifold $(M,g,\varphi)$ defined by a unit spinor field $\varphi\in S$ is harmonic if and only if
\begin{equation*}
\Delta\varphi=-\frac{1}{2}c_{\xi}\cdot\varphi=-\frac{1}{2}\sigma_{\xi}\cdot\varphi+\frac{3}{4}|\xi|^2\varphi,
\end{equation*}
where $\xi$ is the intrinsic torsion.
\end{prop}

By the Lichnerowicz formula we have an immediate corollary.
\begin{cor}\label{cor:general}
Assume that the defining unit spinor $\varphi$ is harmonic, i.e., $\varphi$ is in the kernel of the Dirac operator. Then, a $G$--structure is harmonic if and only if $\sigma_{\xi}$, equivalently $c_{\xi}$, acts on $\varphi$ as a scalar. In this in the case, the scalar corresponding to $\sigma_{\xi}$ is
\begin{equation*}
-\frac{1}{2}{\rm Scal}^g+\frac{3}{2}|\xi|^2,
\end{equation*}
where ${\rm Scal}^g$ is the scalar curvature of $g$.
\end{cor} 
\begin{proof}
Assume that a given $G$--structure in harmonic. Then, by above considerations, $\Delta\varphi=-\frac{1}{2}c_{\xi}\cdot\varphi$ and by \eqref{eq:cTandsigmaT} the second equality holds. Moreover, by the Lichnerowicz formula $\Delta\varphi=-\frac{1}{2}{\rm Scal}^g\varphi$, thus the second part follows. 

Conversely, if $c_{\xi}$ acts as a scalar, say $C$, then by the Lichnerowicz formula and above considerations
\begin{equation*}
-\frac{1}{2}{\rm Scal}^g\varphi=\Delta\varphi=-\frac{1}{2}L\cdot\varphi-\frac{1}{2}C\varphi.
\end{equation*}
Hence, $L\cdot\varphi=({\rm Scal}^g-C)\varphi$. Since $L\cdot\varphi$ is either orthogonal to $\varphi$ or $0$, it must be $0$, so $L=0$.
\end{proof}

In the following sections we obtain the characterization of harmoniciy of $SU(3)$ and $G_2$--structures using the spinorial approach in \cite{ACFH}. We show the relations with the above general approach and state appropriate examples.

\section{$SU(3)$--structures}

Let $\rho:{\rm Spin}(6)\to\Delta=\Delta_6=\mathbb{R}^8$ be the real spin representation. It can be realized in the following way \cite{ACFH}
\begin{align}\label{eq:spinrepresentation}
& e_1=+E_{18}+E_{27}-E_{36}-E_{45}, && e_2=-E_{17}+E_{28}+E_{35}-E_{46}, \notag\\
& e_3=-E_{16}+E_{25}-E_{38}+E_{47}, && e_4=-E_{15}-E_{26}-E_{37}-E_{48},\\
& e_5=-E_{13}-E_{24}+E_{57}+E_{68}, && e_6=+E_{14}-E_{23}-E_{58}+E_{67},\notag
\end{align} 
where $E_{ij}$ is a skew--symmetric matrix such that $E_{ij}e_j=-e_i$. Moreover, let $j=e_1\cdot e_2\cdot\ldots\cdot e_6$. Acting on spinors, $j:\Delta\to\Delta$ is an almost complex structure anti--commuting with the Clifford multiplication by vectors. The crucial observation in \cite{ACFH} is that for a fixed unit spinor $\varphi\in \Delta$ we have the following orthogonal decomposition
\begin{equation*}
\Delta={\rm span}\{\varphi\}\oplus{\rm span}\{j\cdot\varphi\}\oplus\{X\cdot \varphi\mid X\in\mathbb{R}^6\}.
\end{equation*}
Such a spinor defines a group $SU(3)\subset SO(6)$ in a sense that $SU(3)$ is a stabilizer the of $\varphi$, or equivalently, the Lie algebra that anihilates $\varphi$ is $\alg{su}(3)$. Moreover,
\begin{equation}\label{eq:algmphi}
\begin{split}
\alg{su}(3)^{\bot}\cdot\varphi &={\rm span}\{\varphi\}^{\bot},\\
\alg{u}(3)^{\bot}\cdot\varphi &={\rm span}\{\varphi,j\cdot\varphi\}^{\bot}=\{X\cdot \varphi\mid X\in\mathbb{R}^6\}.
\end{split}
\end{equation}  

\begin{exa}\label{ex:0}
Let us demonstrate above formulas on the appropriate example. Choose $\varphi=(0,0,0,0,1,0,0,0)=s_5\in\Delta$. Such a choice is determined by an Examples \ref{ex:1} and \ref{ex:2} from the last section. Then, simple calculations show that, by a realization of the spin representation \eqref{eq:spinrepresentation},
\begin{equation*}
\{X\cdot\varphi\mid X\in\mathbb{R}^6\}={\rm span}\{s_1,s_2,s_3,s_4,s_7,s_8\},\quad j\cdot\varphi=s_6
\end{equation*}
and the Lie algebra $\alg{su}(3)$ of the anihilator of the spinor $\varphi$ is generated by
\begin{equation*}
e_{13}-e_{24},\, e_{14}+e_{23},\, e_{15}+e_{26},\, e_{16}-e_{25},\, e_{35}-e_{46},\, e_{36}+e_{45},\, e_{12}+e_{34},\, e_{34}+e_{56}.
\end{equation*}
Hence, $\alg{m}=\alg{su}(3)^{\bot}$ is a span of
\begin{equation*}
e_{35}+e_{46},\, e_{36}-e_{45},\,e_{15}-e_{26},\, e_{16}+e_{25},\, e_{13}+e_{24},\, e_{14}-e_{23},\, e_{12}-e_{34}+e_{56}.
\end{equation*}
\end{exa}

Let $(M,g)$ be a $6$--dimensional spin manifold with a unit spinor (field) $\varphi$. Since the stabilizer in $SO(6)$ of a unit spinor is $SU(3)$ we get the existence of $SU(3)$--structure on $M$ \cite{ACFH}. This induces the splitting of the real spinor bundle $S$ and implies existence of an emdomorphism $S\in{\rm End}(TM)$ and a one form $\eta$ such that \cite{ACFH}
\begin{equation}\label{eq:SU3mainrel}
\nabla^g_X\varphi=S(X)\cdot\varphi+\eta(X)j\cdot\varphi.
\end{equation}
By \eqref{eq:SU3mainrel} and \eqref{eq:mainstartrelation} we have
\begin{align*}
\frac{1}{2}\nabla^g_Y(\xi_X\cdot\varphi) &=\frac{1}{2}\left((\nabla^g_Y\xi_X)\cdot\varphi+\xi_X\cdot\nabla^g_Y\varphi\right)\\
&=\frac{1}{2}\left((\nabla^g_Y\xi_X)\cdot\varphi+\xi_X\cdot S(Y)\cdot\varphi+\eta(Y)\xi_X\cdot j\cdot\varphi\right)
\end{align*}
and, on the other hand, since $j$ is $\nabla^g$--parallel,
\begin{align*}
\frac{1}{2}\nabla^g_Y(\xi_X\cdot\varphi) &=(\nabla^g_Y S)(X)\cdot\varphi+S(\nabla^g_YX)\cdot\varphi+S(X)\cdot S(Y)\cdot\varphi\\
&+\eta(Y)S(X)\cdot j\cdot\varphi+(\nabla^g_Y\eta)(X)j\cdot\varphi+\eta(\nabla^g_YX)j\cdot\varphi\\
&+\eta(X)j\cdot S(Y)\cdot\varphi-\eta(X)\eta(Y)\varphi.
\end{align*}
Hence, comparing both sides with $X=Y=e_i$ and taking into account the equality
\begin{equation*}
\frac{1}{2}\xi_{\nabla^g_{e_i}e_i}\cdot\varphi=S(\nabla^g_{e_i}e_i)\cdot\varphi+\eta(\nabla^g_{e_i}e_i)\varphi.
\end{equation*}
we get
\begin{align*}
\frac{1}{2}L\cdot\varphi &=-\frac{1}{2}\sum_i \xi_{e_i}\cdot S(e_i)\cdot\varphi-\frac{1}{2}\xi_{\eta^{\sharp}}\cdot j\cdot\varphi+({\rm div}S)\cdot\varphi-|S|^2\varphi+S(\eta^{\sharp})\cdot j\cdot\varphi\\
&+{\rm div}(\eta^{\sharp})j\cdot \varphi+j\cdot S(\eta^{\sharp})\cdot\varphi-|\eta|^2\varphi.
\end{align*}
Applying \eqref{eq:cliffordmultiplication} we obtain
\begin{equation*}
\frac{1}{2}L\cdot\varphi=\chi^S\cdot\varphi-\frac{1}{2}\xi_{\eta^{\sharp}}\cdot j\cdot\varphi+({\rm div}S)\cdot\varphi+{\rm div}(\eta^{\sharp})j\cdot \varphi+j\cdot S(\eta^{\sharp})\cdot\varphi-|\eta|^2\varphi,
\end{equation*}
where $\chi^S$ is a vector field given by
\begin{equation*}
\chi^S=\sum_i \xi_{e_i}S(e_i).
\end{equation*}

We may state and prove the main theorem of this section. Before that, let us say few words about Gray--Hervella classes of possible $SU(3)$--structures and state some additional simple observations.

In general, for any $U(n)$--structure the intrinsic torsion belongs to the space $T^{\ast}M\otimes \alg{u}(n)^{\bot}(TM)$. Under the natural action of $U(n)$ it splits into four modules, so called Gray--Hervella classes \cite{GH}, $\mathcal{W}_1\oplus\mathcal{W}_2\oplus\mathcal{W}_3\oplus\mathcal{W}_4$. For $SU(n)$ we have one additional class $\mathcal{W}_5$, which corresponds to the one form $\eta$. The case $n=3$ is special. Each module $\mathcal{W}_1$ and $\mathcal{W}_2$ split into two modules $\mathcal{W}^{\pm}_i$, $i=1,2$. Therefore, we have the following splitting
\begin{equation*}
T^{\ast}M\otimes\alg{su}(3)^{\bot}(TM)=\mathcal{W}^+_1\oplus\mathcal{W}^-_1\oplus\mathcal{W}^+_2\oplus\mathcal{W}^-_2\oplus\mathcal{W}_3\oplus\mathcal{W}_4\oplus\mathcal{W}_5.
\end{equation*}
Each class has a nice interpretation in terms of $S$ and $\eta$ (see \cite{ACFH}):
\begin{align*}
\mathcal{W}^+_1 &:\quad S=\lambda\, J_{\varphi},\quad \eta=0,\\
\mathcal{W}^-_1 &:\quad S=\mu\, {\rm Id},\quad \eta=0,\\
\mathcal{W}^+_2 &:\quad S\in\alg{su}(3),\quad \eta=0,\\
\mathcal{W}^-_2 &:\quad S\in{\rm Sym}^2_0(T^{\ast}M), SJ_{\varphi}=J_{\varphi}S,\quad\eta=0,\\
\mathcal{W}_3 &:\quad S\in{\rm Sym}^2_0(T^{\ast}M), SJ_{\varphi}=-J_{\varphi}S,\quad\eta=0,\\
\mathcal{W}_4 &:\quad S\in\Lambda^2(T^{\ast}M), SJ_{\varphi}=-J_{\varphi}S,\quad\eta=0,\\
\mathcal{W}_5 &:\quad S=0,
\end{align*}
where $\lambda,\mu$ are constants and $J_{\varphi}$ is an almost complex structure induced by $\varphi$ \cite{ACFH},
\begin{equation*}
J_{\varphi}(X)\cdot\varphi=j\cdot X\cdot\varphi,\quad X\in TM.
\end{equation*}

Assume for a while that $\eta=0$. From the definition of $J_{\varphi}$ we immediately get
\begin{equation}\label{eq:nablaJphi}
(\nabla^g_YJ_{\varphi})(X)\cdot\varphi=2S(Y)\cdot J_{\varphi}(X)\cdot\varphi+2g(J(X),S(Y))\varphi-2g(X,S(Y))j\cdot\varphi,
\end{equation}
which implies
\begin{equation}\label{eq:divJphi}
({\rm div}J_{\varphi})\cdot\varphi=2\sum_i S(e_i)\cdot J_{\varphi}(e_i)\cdot\varphi-2({\rm tr}S) j\cdot\varphi-2{\rm tr}(J_{\varphi}S)\varphi.
\end{equation}

The following lemma shows that in many cases a vector field $\chi^S$ vanishes.
\begin{lem}\label{lem:chisvanishes}
If $\eta=0$, then $\chi^S=0$.
\end{lem} 
\begin{proof}[First proof]
Since $\eta=0$, the intrinsic torsion $\xi$ may be described as $g(\xi_YY,Z)=\psi_{\varphi}(S(X),Y,Z)$, where $\psi_{\varphi}$ is a $3$--form induced by $\varphi$, $\psi_{\varphi}(X,Y,Z)=-\skal{X\cdot Y\cdot Z\cdot \varphi}{\varphi}$ \cite{ACFH}. Thus $\xi_X S(X)=0$ for any $X\in TM$. In particular, $\chi^S$ vanishes. 
\end{proof}
\begin{proof}[Second proof]
By the assumption $\eta=0$, the intrinsic torsion is in fact the intrinsic torsion of corresponding $U(3)$--structure. It is well known that in this case
\begin{equation*}
\xi_XY=-\frac{1}{2}J_{\varphi}(\nabla^g_XJ_{\varphi})Y.
\end{equation*}
Thus, by \eqref{eq:nablaJphi} and the definition of $J_{\varphi}$
\begin{align*}
2J_{\varphi}(\chi^S) &=\sum_i (\nabla^g_{e_i}J_{\varphi})S(e_i)\\
&=2\sum_i S(e_i)\cdot J_{\varphi}(S(e_i))-2|S|^2j\cdot\varphi\\
&=2\sum_i S(e_i)\cdot j\cdot S(e_i)\cdot\varphi-2|S|^2j\cdot\varphi\\
&=0.
\end{align*}
Hence $\chi^S=0$.
\end{proof}

The main theorem of this section reads as follows.
\begin{thm}\label{thm:SU3harmonic}
A $SU(3)$--structure on a $6$--dimensional spin manifold induced by a unit spinor $\varphi$ is harmonic if and only if the following condition holds
\begin{equation}\label{eq:harmonicSU3}
\chi^S\cdot\varphi-\frac{1}{2}\xi_{\eta^{\sharp}}\cdot j\cdot\varphi+({\rm div}S)\cdot\varphi
+{\rm div}(\eta^{\sharp})j\cdot \varphi+j\cdot S(\eta^{\sharp})\cdot\varphi-|\eta|^2\varphi=0.
\end{equation}
If $\eta=0$, then harmonicity is equivalent to ${\rm div}S=0$. In particular, if the intrinsic torsion $\xi$ belongs to the $\mathcal{W}^+_1\oplus\mathcal{W}^-_1$ class or to the pure class $\mathcal{W}_3$, then a $SU(3)$--structure is harmonic.
\end{thm}

\begin{proof}
The only thing which is left to prove is harmonicity of $SU(3)$--structure belonging to the mentioned classes.

$\mathcal{W}^+_1\oplus\mathcal{W}^-_1$ case: In this case, $S=\lambda J_{\varphi}+\mu {\rm Id}$ for two constants $\lambda,\mu$. Then, ${\rm div}(S)=\lambda{\rm div} J_{\varphi}$, thus it suffices to show that the divergence of $J_{\varphi}$ vanishes, but this follows immediately by \eqref{eq:divJphi}.

$\mathcal{W}_3$ case: Since $S$ is symmetric and traceless, it follows that $\sum_i e_i\cdot S(e_i)\cdot\varphi=0$. Thus $D\varphi=0$, where $D$ is the Dirac operator. By Corollary \ref{cor:general}, it suffices to prove that $c_{\xi}$ acts on $\varphi$ by a scalar. Using \eqref{eq:cliffordmultiplication} we obtain
\begin{align*}
2c_{\xi}\cdot\varphi &=\sum_i \xi_{e_i}\cdot\xi_{e_i}\cdot\varphi
=2\sum_i \xi_{e_i}\cdot S(e_i)\cdot\varphi\\
&=2\sum_i (S(e_i)\cdot\xi_{e_i}\cdot\varphi+2\chi^S\cdot\varphi)
=-4|S|^2\varphi,
\end{align*}
since by Lemma \ref{lem:chisvanishes}, $\chi^S$ vanishes.
\end{proof}

\begin{rem}
The harmonicity condition in Theorem \ref{thm:SU3harmonic} can be stated in, maybe, more elegant way, however, less applicable for our further considerations (see the section concerning examples). Namely, introduce an operation (commutator) $[\omega,\tau]$, for any $2$--forms $\omega$ and $\tau$by (compare \cite{PN} and \cite{Ag0})
\begin{equation*}
[\omega,\tau]=\sum_i (e_i\lrcorner\omega)\wedge (e_i\lrcorner\tau).
\end{equation*}  
Then $[\omega,\tau]$ is a $2$--form with the following action on spinors 
\begin{equation*}
\omega\cdot\tau-\tau\cdot\omega=2[\omega,\tau].
\end{equation*} 
Since the element $j$ corresponds to $J_{\varphi}$ treated as a K\"ahler form, we have
\begin{equation*}
2[\xi_{\eta^{\sharp}},J_{\varphi}]\cdot\varphi=\xi_{\eta^{\sharp}}\cdot j\cdot\varphi-2j\cdot S(\eta^{\sharp})\cdot\varphi+2|\eta|^2\varphi. 
\end{equation*}
Hence, condition \eqref{eq:harmonicSU3} reads as
\begin{equation*}
\chi^S\cdot\varphi-[\xi_{\eta^{\sharp}}, J_{\varphi}]\cdot\varphi+({\rm div}S)\cdot\varphi
+{\rm div}(\eta^{\sharp})j\cdot \varphi=0.
\end{equation*}
\end{rem}

\begin{rem}
The fact that in general the $U(n)$--structure of Gray--Hervella pure classes $\mathcal{W}_1$ or $\mathcal{W}_3$ or $\mathcal{W}_4$ is harmonic was proved, without spinorial approach, in \cite{GDMC}. Our approach is based only on the definitions of $J_{\varphi}$ and $S$ by spinorial approach in \cite{ACFH}. We managed only to show that for $\mathcal{W}_1$ and $\mathcal{W}_3$ we have harmonic structures. In the case $\mathcal{W}_4$ we are only able to establish the correspondence with the classical definition of this class. Let us enlarge on this.

We have that $S$ is skew--symmetric and $SJ_{\varphi}=-J_{\varphi}S$. In other words, $S\in{\rm u}(3)^{\bot}(TM)$. Hence, see \eqref{eq:algmphi}, there is a unique vector field $Z_{\varphi}$ such that
\begin{equation}\label{eq:Leefield}
S\cdot\varphi=Z_{\varphi}\cdot\varphi.
\end{equation}
Let us first describe $Z_{\varphi}$. In this case, \eqref{eq:divJphi} reduces to
\begin{equation}\label{eq:divJphi2}
({\rm div}J_{\varphi})\cdot\varphi=2\sum_i S(e_i)\cdot J_{\varphi}(e_i)\cdot\varphi.
\end{equation} 
Moreover, $\sum_i S(e_i)\cdot e_i=2S$ as acting on spinors. Applying $j$ to both sides, and using \eqref{eq:Leefield}, we have
\begin{equation*}
\sum_i S(e_i)\cdot J_{\varphi}(e_i)\cdot\varphi=-2 j\cdot S\cdot\varphi=-2 J_{\varphi}(Z_{\varphi}),
\end{equation*}
which by \eqref{eq:divJphi2} implies
\begin{equation*}
J({\rm div} J)=4Z_{\varphi}.
\end{equation*}
Therefore, $4Z_{\varphi}$ is a Lee vector field of a locally conformally K\"ahler structure \cite{Va1}. Differentiating the relation \eqref{eq:Leefield} in the direction of $e_i$ and then multiplying by $e_i$ we obtain after some tedious calculations
\begin{equation*}
\sum_i e_i\cdot(\nabla_{e_i}S)\cdot\varphi-\sum_i e_i\cdot(\nabla_{e_i}Z_{\varphi})\cdot\varphi=-4|Z_{\varphi}|^2\varphi-6S(Z_{\varphi})\cdot\varphi+2|S|^2\varphi.
\end{equation*} 
On the other hand, by \eqref{eq:cliffordmultiplication}, $2S(e_i)\cdot\varphi=e_i\cdot Z_{\varphi}\cdot\varphi-S\cdot e_i\cdot\varphi$, which implies
\begin{equation*}
\sum_i e_i\cdot(\nabla_{e_i}S)\cdot\varphi-\sum_i e_i\cdot(\nabla_{e_i}Z_{\varphi})\cdot\varphi=4|Z_{\varphi}|^2\varphi+6S(Z_{\varphi})\cdot\varphi.
\end{equation*}
Comparing last two relations we get
\begin{equation*}
S(Z_{\varphi})=0\quad\textrm{and}\quad |S|^2=4|Z_{\varphi}|^2.
\end{equation*}
Thus, by \eqref{eq:nablaJphi} we have
\begin{equation*}
(\nabla_X J_{\varphi})(J_{\varphi}Z_{\varphi})\cdot\varphi=-2S(X)\cdot Z_{\varphi}\cdot\varphi.
\end{equation*}
The operator $X\mapsto (\nabla_X J_{\varphi})(J_{\varphi}Z_{\varphi})$ is skew--symmetric, hence can be considered as a $2$--form. Up to a constant factor this form has been considered by Vaisman \cite{Va2}.
\end{rem}

\section{$G_2$--structures}

Analogously as in dimension $6$, the spin representation $\rho:{\rm Spin}(7)\to\Delta=\Delta_7=\mathbb{R}^8$ is real and can be realized identically as in the $6$--dimensional case with additional action of $e_7$ given by
\begin{equation*}
e_7=+E_{12}-E_{34}-E_{56}+E_{78}.
\end{equation*}
Fix a unit spinor $\varphi\in\Delta$. By dimensional reasons we have
\begin{equation*}
\Delta={\rm span}\{\varphi\}\oplus\{X\cdot\varphi\mid X\in\mathbb{R}^7\}.
\end{equation*}

\begin{exa}\label{ex:00}
Analogously, as in the $SU(3)$ case, let us do some calculations in a concrete example. Choose a unit spinor $\varphi=s_5\in \Delta$. Then, the Lie algebra, which anihilates $\varphi$ via Clifford multiplication of $2$--forms via the realization of the real spin representation is spanned by elements
\begin{align*}
& e_{16}+e_{37},\quad e_{16}-e_{25},\quad e_{15}+e_{26},\quad e_{26}+e_{47},\quad e_{17}-e_{36},\quad e_{17}+e_{45},\quad e_{27}-e_{35},\\
& e_{27}-e_{46},\quad e_{12}+e_{34},\quad e_{12}-e_{56},\quad e_{13}-e_{24},\quad e_{13}-e_{67},\quad e_{14}+e_{23},\quad e_{14}+e_{57}.
\end{align*}
Hence it is $\alg{g}_2$. Moreover, its orthogonal complement $\alg{m}=\alg{g}_2^{\bot}$ in $\alg{so}(7)$ is spanned by
\begin{align*}
& e_{16}-e_{37}+e_{25},\quad e_{15}-e_{26}+e_{47},\quad e_{17}+e_{36}-e_{45},\quad e_{27}+e_{35}+e_{46},\\
& e_{12}-e_{34}+e_{56},\quad e_{13}+e_{24}+e_{67},\quad e_{14}-e_{23}-e_{57}.
\end{align*}
\end{exa}

Let $(M,g)$ be a $7$--dimensional spin manifold with the spinor bundle $S$ and a unit spinor (field) $\varphi$. Above decomposition induces a splitting of the spinor bundle and implies the existence of an endomorphism $S\in{\rm End}(TM)$ such that
\begin{equation}
\nabla_X\varphi=S(X)\cdot\varphi.
\end{equation}
Since a stabilizer in $SO(7)$ of a unit spinor  is $G_2$, $(M,g,\varphi)$ becomes a $G_2$--structure \cite{ACFH}. The harmonicity condition, or more generally, the formula for a tensor $L$ becomes, just putting $\eta=0$ in the $SU(3)$ case,
\begin{equation*}
\frac{1}{2} L\cdot\varphi=({\rm div}S)\cdot\varphi.
\end{equation*}
Notice, that in $G_2$ case the vector field $\chi^S$ vanishes, since the intrinsic torsion may be described as follows $g(\xi_XY,Z)=\frac{2}{3}\psi_{\varphi}(S(X),Y,Z)$ with a $3$--form $\psi_{\varphi}$ defined in the same way as in the $SU(3)$ case. Thus we have the following observation.
\begin{thm}\label{thm:G2harmonic}
A $G_2$--structure on a spin $7$--dimensional manifold is harmonic if and only if ${\rm div}S=0$.
\end{thm}

For a $G_2$--structure the space of all possible intrinsic torsions $T^{\ast}M\otimes\alg{g}_2^{\bot}(TM)$ splits into four irreducible modules $\mathcal{W}_1,\ldots,\mathcal{W}_4$ \cite{ACFH}:
\begin{align*}
\mathcal{W}_1 & :\quad S=\lambda{\rm Id},\\
\mathcal{W}_2 & :\quad S\in\alg{g}_2,\\
\mathcal{W}_3 & :\quad S\in{\rm Sym}^2_0(T^{\ast}M),\\
\mathcal{W}_4 & :\quad S=V\lrcorner\psi_{\varphi},\quad V\in TM,
\end{align*}
where $\psi_{\varphi}$ is a $3$--form defined as $\pi_{\varphi}(X,Y,Z)=\skal{X\cdot Y\cdot Z\cdot\varphi}{\varphi}$.  The class for which the condition of harmonicity may be explicitly described is $\mathcal{W}_1$ defined by the condition $S=\lambda{\rm Id}$, where $\lambda$ is a smooth function. Then ${\rm div}S={\rm grad}(\lambda)$. Hence, a $G_2$--structure in $\mathcal{W}_1$ class is harmonic if and only if $\lambda$ is constant. Moreover, by the same lines as in the proof of Theorem \ref{thm:SU3harmonic}, we see that a $G_2$--structure belonging to a $\mathcal{W}_3$ class is harmonic. Hence wa may state the following fact.

\begin{cor}\label{cor:G2str}
A $G_2$--structure belonging to pure class
\begin{enumerate}
\item $\mathcal{W}_1$ is harmonic if and only if $\lambda$ is a constant, i.e., a $G_2$--structure is nearly parallel, 
\item $\mathcal{W}_3$ is harmonic.
\end{enumerate} 
\end{cor}

\begin{rem}\label{rem:Agricola}
Some of the results, especially second part of Theorem \ref{thm:SU3harmonic} and Corollary \ref{cor:G2str}(1) may be obtained with an alternative approach communicated to the author by I. Agricola. Namely, assuming that a $SU(3)$ or a $G_2$--structure admits a characteristic connection $\nabla^c$ (see \cite[p. 45]{Ag0} for a definition), which holds for considered structures excluding $\mathcal{W}_2$ cases \cite{NS,FI}, the harmonicity condition by \cite[Theorem 3.7]{GDMC} is equivalent to $\delta T^c=0$, where $T^c$ is a torsion of a characteristic connection (characteristic torsion). In particular, if $\nabla^c T^c=0$, then a considered structure is harmonic. It is known that nearly K\"ahler and nearly parallel $G_2$--structures admit parallel characteristic connection \cite{VK,TF}, thus are harmonic as $U(3)$ and $G_2$--structures, respectively.
\end{rem}

\section{Examples}

We justify described above theory on appropriate examples. We begin with a suitable introduction. We rely on \cite{BFGK}. Consider a homogeneous space $M=K/H$, where $K$ is a compact, connected Lie group and $H$ its closed subgroup. Denote by $\alg{k}$ and $\alg{h}$ the lie algebras of $K$ and $H$, respectively. Assume we have a decomposition $\alg{k}=\alg{h}\oplus\alg{n}$, where $\alg{n}$ is an orthogonal complement of $\alg{h}$ with respect to and ${\rm ad}(H)$--invariant positive bilinear form ${\bf B}$ on $\alg{k}$. Then ${\bf B}$ induces a Riemannian metric $g$ on $M$. By a well known theorem by Wang the Levi--Civita connection $\nabla^g$ is identified with an invariant linear map $\Lambda:\alg{n}\to\alg{so}(\alg{n})$. 

Denote by $\lambda:H\to{\rm SO}(\alg{n})$ the isotropy representation. Notice, that a tangent bundle of $M$ may be described as $TM=K\times_{\lambda}\alg{n}$ and hence any tensor bundle $T^{\otimes k}M$ equals $K\times_{\lambda}\alg{n}^{\otimes k}$, etc. Consider additionally a $G$--structure on $M$. Then we have a splitting $\alg{so}(\alg{n})=\alg{g}\oplus\alg{m}$. The intrinsic torsion $\xi$ is a section of a bundle $K\times_{\lambda}(\alg{n}^{\ast}\otimes\alg{m})$. Since, there is a bijection between sections of the associated bundle $K\times_{\tau}V$, for a representation $\tau:H\to{\rm End}(V)$ and $\tau$--invariant functions $f:K\to V$, it follows that $\xi$ may be considered as an invariant function $f_{\xi}:K\to\alg{n}^{\ast}\otimes\alg{m}$,
\begin{equation}\label{eq:fxi}
f_{\xi}(g)={\rm ad}(g^{-1})\Lambda_{\alg{m}},
\end{equation}
where $\Lambda_{\alg{m}}$ is a $\alg{m}$--component of $\Lambda$.  

Following \cite{BFGK}, let us introduce a spin structure on $M$. Assume that there is the lift $\tilde{\lambda}:H\to {\rm Spin}(\alg{n})$, i.e., $\pi\circ\tilde{\lambda}=\lambda$, where $\pi:{\rm Spin}(\alg{n})\to SO(\alg{n})$ is a double covering. Then $M$ admits a spin structure, namely ${\rm Spin}(M)=K\times_{\tilde{\lambda}}{\rm Spin}(\alg{n})$. The connection on the spinor bundle $S=K\times_{\rho\tilde{\lambda}}\Delta$, where $\rho:{\rm Spin}(\alg{n})\to{\rm End}(\Delta)$ is a real spin representation, induced from the Levi--Civita connection is therefore identified with an invariant linear map $\tilde{\Lambda}:\alg{n}\to\alg{spin}(\alg{n})$ via the correspondence
\begin{equation*}
\pi_{\ast}\circ\tilde{\Lambda}=\Lambda.
\end{equation*}

Assume there is a isotropy invariant unit spinor $\varphi_0$. Thus, as a constant function $f_{\varphi}(k)=\varphi_0$, it induces a global unit spinor $\varphi$ on the spinor bundle $S$ over $M$. Then, by the Ikeda result \cite{Ik}, $\nabla_X\varphi$ corresponds to
\begin{equation*}
f_X(f_{\varphi})+\rho_{\ast}(\tilde{\Lambda}(f_X))f_{\varphi}.
\end{equation*}
Since $f_{\varphi}$ is constant, the first element vanishes. Therefore, the spinorial laplacian of $\varphi$ corresponds to
\begin{equation}\label{eq:spinorlaplacianhomog}
-\sum_i \rho_{\ast}(\tilde{\Lambda}(E_i))\rho_{\ast}(\tilde{\Lambda}(E_i))\varphi_0.
\end{equation}
Moreover, the element $c_{\xi}$ by a definition equals
\begin{equation}\label{eq:cxihomog}
-\frac{1}{2}\sum_i\rho_{\ast}(\tilde{\Lambda}_{\alg{m}}(E_i))\rho_{\ast}(\tilde{\Lambda}_{\alg{m}}(E_i))\varphi_0.
\end{equation} 

\begin{prop}\label{prop:homog}
Let $M=K/H$ be a reductive homogeneous space with a spin structure induced by the lift $\tilde{\lambda}$ of the isotropy representation $\lambda$. Assume that a $G$--structure on $M$ is induced by a spinor, which is a fixed point of $\lambda$. If a minimal connection $\nabla^G$ is induced by a zero map $\Lambda_{\alg{g}}:\alg{n}\to\alg{g}$, where $\alg{k}=\alg{h}\oplus\alg{n}$ is a reductive decomposition (i.e., is a canonical connection), then a $G$--structure is harmonic. 
\end{prop}
\begin{proof}
Follows immediately by Proposition \ref{prop:general}, relations  \eqref{eq:spinorlaplacianhomog} and \eqref{eq:cxihomog} and the fact that $\alg{g}$--component of $\Lambda:\alg{n}\to\alg{so}(\alg{n})$ vanishes.
\end{proof}

Now, we deform ${\bf B}$ to ${\bf B}_t$, $t>0$, in the following way: we assume that there is a ${\bf B}$--orthogonal splitting $\alg{n}=\alg{n}_0\oplus\alg{n}_1$. Then we put
\begin{equation*}
{\bf B}_t={\bf B}|_{\alg{n}_0\times\alg{n}_0}+2t{\bf B}|_{\alg{n}_1\times\alg{n}_1}.
\end{equation*}
${\bf B}_t$ induces one parameter family of Riemannian metrics $g_t$ on $M$. Below, we discus the behavior of $(M,g_t,\varphi)$ in three cases.

\begin{exa}\label{ex:1}
Consider a complex projective space $M=\mathbb{CP}^3$, which was studied in detail in \cite{BFGK}. Here, we review all necessary facts and develop these which are indispensable for our purposes. Recall, that $\mathbb{CP}^3$ is a homogeneous space of the form $SO(5)/U(2)$. On the level of Lie algebras, $\alg{so}(5)=\alg{u}(2)\oplus\alg{n}$, where  
\begin{align*}
\alg{u}(2) &={\rm span}\{E_{12}, E_{34}, E_{13}-E_{24}, E_{14}+E_{23}\},\\ \alg{n} &={\rm span}\{E_{15}, E_{25}, E_{35}, E_{45}, E_{13}+E_{24}, E_{14}-E_{23}\}.
\end{align*}
Moreover, decompose $\alg{n}$ into $\alg{n}_0\oplus\alg{n}_1$, where
\begin{equation*}
\alg{n}_0={\rm span}\{E_{15}, E_{25}, E_{35}, E_{45}\},\quad \alg{n}_1={\rm span}\{E_{13}+E_{24}, E_{14}-E_{23}\}.
\end{equation*}
The orthonormal basis of $\alg{n}$ with respect to ${\bf B}_t$ (here ${\bf B}$ is (the negative of) the Killing form) can be chosen in the following way
\begin{align*}
& X_1=E_{15}, && X_2=E_{25}, && X_3=E_{35},\\
& X_4=E_{45}, && X_5=\frac{1}{2\sqrt{t}}(E_{13}+E_{24}), && X_6=\frac{1}{2\sqrt{t}}(E_{14}-E_{23}).
\end{align*} 
Thus we have an identification $\alg{n}=\mathbb{R}^6$. The Levi--Civita connection $\Lambda_t$ takes the form
\begin{align*}
&\Lambda_t(X_1)=\frac{\sqrt{t}}{2}(E_{35}+E_{46}), &&\Lambda_t(X_2)=\frac{\sqrt{t}}{2}(E_{45}-E_{36}),\\
&\Lambda_t(X_3)=\frac{\sqrt{t}}{2}(E_{26}-E_{15}), &&\Lambda_t(X_4)=\frac{\sqrt{t}}{2}(-E_{16}-E_{25}),\\ 
&\Lambda_t(X_5)=\frac{1-t}{2\sqrt{t}}(E_{13}+E_{24}), &&\Lambda_t(X_6)=\frac{1-t}{2\sqrt{t}}(E_{14}-E_{23}).
\end{align*}

The (differential of the) isotropy representation $\lambda_{\ast}:\alg{u}(2)\to \alg{so}(6)$ is of the form
\begin{align*}
&\lambda_{\ast}(E_{12})=E_{12}-E_{56}, &&\lambda_{\ast}(E_{34})=E_{34}+E_{56},\\
&\lambda_{\ast}(E_{13}-E_{24})=E_{13}-E_{24}, &&\lambda_{\ast}(E_{14}+E_{23})=E_{14}+E_{23}.
\end{align*}
It can be shown \cite{BFGK} that $\lambda$ has a lift to a map $\tilde{\lambda}:U(2)\to{\rm Spin}(6)$, which implies existence of a spin structure on $M$. Via the spin representation \eqref{eq:spinrepresentation} we see that $\tilde{\lambda}_{\ast}$ has the following form
\begin{align*}
&\tilde{\lambda}_{\ast}(E_{12})=E_{34}+E_{78}, &&\tilde{\lambda}_{\ast}(E_{34})=E_{12}-E_{78},\\
&\tilde{\lambda}(E_{13}-E_{24})=-E_{13}+E_{24}, &&\tilde{\lambda}_{\ast}(E_{12}+E_{23})=-E_{14}-E_{23}.
\end{align*}
Thus, spinors $\varphi=s_5$ and $\varphi=s_6$ are anihilated by  above isotropy representation. Thus, each spinor in the span of $s_5$ and $s_6$ defines a global spinor field. Fix a spinor $\varphi=s_5$ defining a $SU(3)$--structure on $M$. 

From Example \ref{ex:0} we see that $\Lambda_t$ has values in $\alg{m}$, hence corresponds to the intrinsic torsion, whereas, the $SU(3)$--connection corresponds to the zero map, $\Lambda_{\alg{su}(3)}\equiv 0$. Thus, by Proposition \ref{prop:homog}, the considered structure is harmonic for all $t>0$.

Let us consider an additional approach. By the definition of $\tilde{\Lambda}$, which corresponds to the induced connection on the spinor bundle, we have $\tilde{\Lambda}(X)\varphi=S(X)\cdot \varphi$, where $S$ is diagonal of the form (compare \cite{ACFH})
\begin{equation*}
S=-{\rm diag}\left( \frac{\sqrt{t}}{2},\frac{\sqrt{t}}{2}, \frac{\sqrt{t}}{2}, \frac{\sqrt{t}}{2}\frac{1-t}{2\sqrt{t}}, \frac{1-t}{2\sqrt{t}}\right),
\end{equation*}
Hence the one--form $\eta$ vanishes. Since
\begin{equation*}
S=-\frac{t+1}{6\sqrt{t}}{\rm Id}-\frac{2t-1}{6\sqrt{t}}{\rm diag}(1,1,1,1,-2,-2),
\end{equation*} 
the $SU(3)$--structure it is of class $\mathcal{W}_1^-\oplus\mathcal{W}_2^-$ for $t\neq\frac{1}{2}$ and $\mathcal{W}_1^-$ for $t=\frac{1}{2}$. In both cases, ${\rm div}S$, which corresponds to $\sum_i (\Lambda(X_i)S)X_i$, vanishes, hence the considered $SU(3)$--structures are harmonic.
\end{exa}

\begin{exa}\label{ex:2}
Consider a Lie group $M={\rm Spin}(4)\equiv \mathbb{S}^3\times\mathbb{S}^3$. Then its Lie algebra is isomorphic to $\alg{so}(4)$. Consider a following decomposition $\alg{so}(4)=\alg{n}_0\oplus\alg{n}_1$, where $\alg{n}_0={\rm span}\{E_{14},E_{24},E_{34}\}$ and $\alg{n}_1={\rm span}\{E_{12},E_{13},E_{23}\}$ \cite{BFGK}. Computing the Lie brackets of generators, we see that $\alg{n}_1$ is a Lie subalgebra. Choose the following orthonormal basis of $\alg{so}(4)$ with respect to ${\bf B}_t$ (here ${\bf B}$ is again the negative of the Killing form):
\begin{align*}
& X_1=E_{14},\quad X_2=E_{24},\quad X_3=E_{34},\\ 
& X_4=\frac{1}{\sqrt{2t}}E_{12},\quad X_5=\frac{1}{\sqrt{2t}}E_{13},\quad X_6=\frac{1}{\sqrt{2t}}E_{23}.
\end{align*}
With this choice, $\alg{so}(4)=\mathbb{R}^6$. The Levi--Civita connection of ${\bf B}_t$ is represented by a map $\Lambda:\mathbb{R}^6\to\alg{so}(6)$ (compare \cite{BFGK})
\begin{align*}
&\Lambda_t(X_1)=\frac{1}{2}\sqrt{2t}E_{24}+\frac{1}{2}\sqrt{2t}E_{35}, &&\Lambda_t(X_2)=-\frac{1}{2}\sqrt{2t}E_{14}+\frac{1}{2}\sqrt{2t}E_{36},\\
&\Lambda_t(X_3)=-\frac{1}{2}\sqrt{2t}E_{15}-\frac{1}{2}\sqrt{2t}E_{26}, &&\Lambda_t(X_4)=\frac{1-t}{\sqrt{2t}}E_{12}+\frac{1}{2\sqrt{2t}}E_{56},\\
&\Lambda_t(X_5)=\frac{1-t}{\sqrt{2t}}E_{13}-\frac{1}{2\sqrt{2t}}E_{46}, &&\Lambda_t(X_6)=\frac{1-t}{\sqrt{2t}}E_{23}+\frac{1}{2\sqrt{2t}}E_{45}
\end{align*}

The spin structure is the trivial one $M\times {\rm Spin}(4)$ and the spinor bundle is, again, the trivial bundle $M\times\Delta$ \cite{BFGK}. Hence, each smooth function $f_{\varphi}:M\to \Delta$ defines a global spinor field. Choose a defining spinor being the constant function equal to $s_5\in\Delta$. Then, the equality $\tilde{\Lambda}(X)s_5=S(X)\cdot s_5+\eta(X)j\cdot s_5$ is satisfied by
\begin{equation*}
S=\left(\begin{array}{cccccc}
-\frac{1}{2}\sqrt{2t} & 0 & 0 & 0 & \frac{1}{2\sqrt{2t}} & 0 \\
0 & \frac{1}{2}\sqrt{2t} & 0 & 0 & 0 & -\frac{1}{2\sqrt{2t}} \\
0 & 0 & 0 & 0 & 0 & 0 \\
0 & 0 & 0 & 0 & 0 & 0 \\
-\frac{1}{2}\sqrt{2t} & 0 & 0 & & -\frac{1-t}{\sqrt{2t}} & 0 \\
0 & \frac{1}{2}\sqrt{2t} & 0 & 0 & 0 & \frac{1-t}{\sqrt{2t}}
\end{array}\right),\quad \eta=\left(\frac{3}{2}-t\right)\frac{1}{\sqrt{2t}}X_4^{\flat}.
\end{equation*}
Hence the considered $SU(3)$--structure is of type $\mathcal{W}_{2}^-\oplus\mathcal{W}_{3}\oplus\mathcal{W}_4\oplus\mathcal{W}_5$. Notice that $j\cdot s_5=s_6$ (see Example \ref{ex:0}). Moreover, it is not hard to check that ${\rm div}S=0$,  $S(\eta^{\sharp})=0$, ${\rm div}(\eta^{\sharp})=0$ and $|\eta^{\sharp}|^2=\frac{1}{2t}\left(\frac{3}{2}-t\right)^2$. Thus harmonicity condition (Theorem \ref{thm:SU3harmonic}) has the following form
\begin{equation}\label{eq:ex2harmonicity}
 \chi^S\cdot s_5-\frac{1}{2}\xi_{\eta^{\sharp}}\cdot s_6-\frac{1}{2t}\left(\frac{3}{2}-t\right)^2s_5=0.
\end{equation}
We need to compute $\xi_{\eta^{\sharp}}\cdot s_5$, which corresponds to $\left(\frac{3}{2}-t\right)\frac{1}{\sqrt{2t}}\Lambda_{\alg{su(3)}^{\bot}}(X_4)$. Since
\begin{equation*}
\Lambda_{\alg{su}(3)^{\bot}}(X_4)=\frac{\sqrt{t}}{2}(-E_{16}-E_{25})_{\alg{su}(3)^{\bot}}=\frac{3-2t}{6\sqrt{2t}}(E_{12}-E_{34}+E_{56}).
\end{equation*}
we see that  
\begin{equation*}
\xi_{\eta^{\sharp}}\cdot s_5=\left(\frac{3}{2}-t\right)^2\frac{1}{6t}s_6.
\end{equation*}
Thus, harmonicity condition \eqref{eq:ex2harmonicity} simplifies to
\begin{equation*}
\chi^S\cdot s_5-\frac{1}{12t}\left(\frac{3}{2}-t\right)^2s_6-\frac{1}{2t}\left(\frac{3}{2}-t\right)^2s_5=0.
\end{equation*}
Since $\chi^S\cdot s_5$ is orthogonal to $s_5$ and $s_6$ we have $t=\frac{3}{2}$ and, in particular, $\eta$ vanishes. Thus, by the fact that ${\rm div}S=0$ and Theorem \ref{thm:SU3harmonic}, the considered $SU(3)$--structure is harmonic only for $t=\frac{3}{2}$. In this case, it is of type $\mathcal{W}_{2}^-\oplus\mathcal{W}_3\oplus\mathcal{W}_4$.
\end{exa}

\begin{exa}\label{ex:3} 
Let $M$ be an Aloff-Wallach space $N(1,1)=SU(3)/\mathbb{S}^1$, where the action $\mathbb{S}^1\to SU(3)$ is by diagonal matrices
\begin{equation*}
\theta\mapsto{\rm diag}(e^{\theta i},e^{\theta i},e^{-2\theta i}).
\end{equation*}
Consider a splitting $\alg{su}(3)=\mathbb{R}\oplus\alg{n}$ such that $\alg{n}=\alg{n}_0\oplus\alg{n}_1$ is given by
\begin{equation*}
\alg{n}_0={\rm span}\{L,E_{12},\tilde{E}_{12}\},\quad \alg{n}_1={\rm span}\{ E_{13},\tilde{E}_{13},E_{23},\tilde{E}_{23}\},
\end{equation*}
where
\begin{equation*}
L={\rm diag}(i,i,0),\quad \tilde{E}_{kl}=iS_{kl}
\end{equation*}
and $S_{kl}$ is a symmetric matrix with $S_{kl}e_l=e_k$. Then, an orthonormal basis of $\alg{n}$ with respect to ${\bf B}_t$, induced from the Killing form, can be chosen as follows
\begin{align*}
& X_1=E_{12},\quad X_2=\tilde{E}_{12},\quad X_3=\frac{1}{\sqrt{2t}}E_{13},\quad X_4=\frac{1}{\sqrt{2t}}\tilde{E}_{13},\\
& X_5=\frac{1}{\sqrt{2t}}E_{23},\quad X_6=\frac{1}{\sqrt{2t}}\tilde{E}_{23},\quad X_7=L.
\end{align*} 
The Levi--Civita connection of $g_t$ gives a map $\Lambda_t:\alg{n}\to\alg{so}(\alg{n})$ \cite{BFGK,AF}:
\begin{align*}
&\Lambda(X_1)=E_{27}-\left(1-\frac{1}{4t}\right)(E_{35}+E_{46}), && \Lambda(X_2)=-E_{17}-\left(1-\frac{1}{4t}\right)(E_{45}-E_{36}),\\
&\Lambda(X_3)=\frac{1}{4t}E_{47}-\frac{1}{4t}(E_{26}-E_{15}), &&\Lambda(X_4)=-\frac{1}{4t}E_{37}+\frac{1}{4t}(E_{16}+E_{25}),\\
&\Lambda(X_5)=-\frac{1}{4t}E_{67}-\frac{1}{4t}(E_{13}+E_{24}), &&\Lambda(X_6)=\frac{1}{4t}E_{57}-\frac{1}{4t}(E_{14}-E_{23}),\\
&\Lambda(X_7)=E_{12}+\left(1-\frac{1}{4t}\right)(E_{34}-E_{56}).
\end{align*}

The isotropy representation $\lambda:\mathbb{S}^1\to{\rm SO}(\alg{n})$ has a lift to a map $\tilde{\lambda}:\mathbb{S}^1\to{\rm Spin}(\alg{n})$, thus there is a spin structure on $M$ \cite{BFGK}. Moreover, the spinor $\varphi_0=s_5$ is a fixed point of this action, hence as a constant function from ${\rm SU}(3)$ to $\Delta$ defines a global spinor field $\varphi$. Consider a $G_2$ structure induced by $\varphi$. By Example \ref{ex:00} the map $\Lambda_t$ takes values in $\alg{m}$ if and only if $t=\frac{1}{8}$. In this case, by Proposition \ref{prop:homog} a $G_2$--structure is harmonic. Let us check harmonicity for remaining values of $t$. It is easy to see that an endomorphism $S$ satisfying $\tilde{\Lambda}(X)\varphi_0=S(X)\cdot\varphi_0$ equals
\begin{align*}
S &=\frac{1}{2}{\rm diag}\left(\left(\frac{1}{2t}-1\right),\left(\frac{1}{2t}-1\right),-\frac{3}{4t},-\frac{3}{4t},-\frac{3}{4t},-\frac{3}{4t},\left(\frac{1}{2t}-1\right)\right).
\end{align*}
In particular, considered $G_2$--structure is of type $\mathcal{W}_1\oplus\mathcal{W}_3$ for $t\neq\frac{5}{4}$ and of pure type $\mathcal{W}_1$ for $t=\frac{5}{4}$. The divergence of $S$, corresponding to $\sum_i \Lambda(X_i)S(X_i)$, vanishes. Hence, for any $t>0$ considered $G_2$--structure is harmonic.
\end{exa}

\end{document}